\documentclass[a4paper]{amsart}

\usepackage{amsmath}
\usepackage{amssymb}
\usepackage{amsthm}
\usepackage{bbm}

\newtheorem{theorem}{Theorem}

\newcommand{\expec}{\mathbf{E}}
\newcommand{\Oh}{\mathrm{O}}
\newcommand{\oh}{\mathrm{o}}
\newcommand{\im}{\mathrm{i}}
\newcommand{\e}{\mathrm{e}}
\newcommand{\dd}{\mathrm{d}}

\title{Counting Finite Languages by Total Word Length}

\author{Stefan Gerhold}

\address{Vienna University of Technology, Wiedner Hauptstra\ss{}e 8--10,
A-1040 Vienna, Austria}

\email{sgerhold at fam.tuwien.ac.at}

\date{\today}

\begin{document}

\begin{abstract}
  We investigate the number of sets of words that can be formed from a finite alphabet,
  counted by the total length of the words in the set.
  An explicit expression for the counting sequence is derived from the generating function, and asymptotics for
  large alphabet respectively large total word length are discussed. Moreover, we derive a Gaussian limit
  law for the number of words in a random finite language.
\end{abstract}



\maketitle

\section{Introduction and Basic Properties}

Let $f_n=f_n(m)$ denote the number of languages (i.e., sets of words) with total word length~$n$ over an alphabet
with $m\geq 2$ symbols~\cite[I.37]{FlSe09}.
For instance, $f_2(2)=5$ and $f_3(2)=16$, as seen from the listings
\[
 \{\mathsf{a},\mathsf{b}\}, \{\mathsf{aa}\}, \{\mathsf{ab}\}, \{\mathsf{ba}\}, \{\mathsf{bb}\}
\]
respectively
\begin{align*}
  &\{\mathsf{a},\mathsf{aa}\}, \{\mathsf{a},\mathsf{ab}\}, \{\mathsf{a},\mathsf{ba}\}, \{\mathsf{a},\mathsf{bb}\},   
    \{\mathsf{b},\mathsf{aa}\}, \{\mathsf{b},\mathsf{ab}\}, \{\mathsf{b},\mathsf{ba}\}, \{\mathsf{b},\mathsf{bb}\}, \{\mathsf{aaa}\}, \\   
  &  \{\mathsf{aab}\},
   \{\mathsf{aba}\}, \{\mathsf{abb}\}, \{\mathsf{baa}\}, \{\mathsf{bab}\}, \{\mathsf{bba}\}, \{\mathsf{bbb}\}.
\end{align*}
Another value is $f_2(3)=12$, illustrated by
\begin{equation*}
  \{\mathsf{aa}\}, \{\mathsf{ab}\}, \{\mathsf{ac}\}, \{\mathsf{ba}\}, \{\mathsf{bb}\}, \{\mathsf{bc}\}, \{\mathsf{ca}\}, \{\mathsf{cb}\},   
    \{\mathsf{cc}\}, \{\mathsf{a},\mathsf{b}\}, \{\mathsf{a},\mathsf{c}\}, \{\mathsf{b},\mathsf{c}\}.
\end{equation*}
The sequence~$f_n(2)$ is number {\bf A102866} of Sloane's On-Line Encyclopedia of Integer Sequences.%
\footnote{http://www.research.att.com/\~{}njas/sequences/}
In the present note, we will derive an explicit expression for~$f_n(m)$ (Theorem~\ref{thm:clf} below),
establish asymptotics (Sections~\ref{se:asympt m}--\ref{se:joint}),
and derive a limit law for the number of words in a random finite language (Section~\ref{se:law}).

The ordinary generating function (ogf)~\cite[I.37]{FlSe09}
\begin{equation}\label{eq:ogf}
  F(z) := \sum_{n=0}^\infty  f_n z^n = \exp \left( \sum_{k=1}^\infty
    \frac{(-1)^{k-1}}{k} \frac{mz^k}{1-mz^k} \right)
\end{equation}
can be obtained by a standard procedure (the ``power set construction''~\cite[I.2]{FlSe09}; finite languages
are sets of sequences built from alphabet elements).
Its first terms are
\begin{equation}\label{eq:first terms}
  F(z) = 1 + mz + \tfrac12 m(3m-1)z^2 + m(\tfrac{13}{6}m^2-\tfrac12 m+\tfrac13) + \Oh(z^4).
\end{equation}
Note that
\[
  F(z) = \exp\left( \frac{mz}{1-mz} \right) \phi(z),
\]
where
\[
  \phi(z;m) = \phi(z) := \exp \left( \sum_{k=2}^\infty \frac{(-1)^{k-1}}{k} \frac{mz^k}{1-mz^k} \right)
\]
is analytic for $|z|<1/\sqrt{m}$.
(Indeed, for $0<\varepsilon<1/\sqrt{m}$, $|z|\leq 1/\sqrt{m}-\varepsilon$, and $k\geq 2$, we have
\begin{equation*}
  m|z|^k \leq m|z|^2 \leq m (m^{-1/2}-\varepsilon)^2
  = 1 - \varepsilon \sqrt{m}(2-\varepsilon \sqrt{m}) =: 1 - \varepsilon',
\end{equation*}
whence
\[
  \left| \frac{mz^k}{1-mz^k} \right| \leq \frac{m|z|^k}{\varepsilon'}.)
\]
The dominating singularity of~$F(z)$ is thus located at~$z=1/m$, leading to the rough
approximation $f_n(m) \approx m^n$.
Clearly (consider languages consisting only of one word), we have $f_n(m)>m^n$ for $m,n\geq2$.
We will see in Theorem~\ref{thm:asympt n} below that the ratio $f_n(m)/m^n$ is~$\e^{2\sqrt{n} + \Oh(\log n)}$.

Our first result is an explicit expression for~$f_n(m)$, which
can be obtained from~\eqref{eq:ogf}. To state it,
we write $\mathbf{i} \vdash\! n$, if the vector $\mathbf{i}=(i_1,\dots,i_n)\in\mathbb{Z}_{\geq0}^n$
represents a partition of~$n$, in the sense
that $i_1 + 2i_2 + \dots + ni_n = n$.
\begin{theorem}\label{thm:clf}
  For $m\geq2$ and $n\geq1$, we have
  \begin{equation}\label{eq:clf}
    f_n(m) = \sum_{\mathbf{i}\, \vdash n} \frac{A_1(m)^{i_1} \dots A_n(m)^{i_n}}{i_1! \dots i_n!},
  \end{equation}
  where
  \[
    A_j(m) := \sum_{d \mid j} (-1)^{d-1} m^{j/d}/d, \qquad j\geq1.
  \]
\end{theorem}
\begin{proof}
  We expand the Lambert series~\cite{Wi06} in the exponent of~$F(z)$, using the geometric series formula:
  \begin{align}
    F(z) &= \exp\left( \sum_{k=1}^\infty \frac{(-1)^{k-1}}{k}
      \sum_{j=1}^\infty m^j z^{kj}  \right) \notag \\
    &= \exp\left( \sum_{n=1}^\infty A_n(m) z^n  \right) \notag \\
    &= 1 + \frac{1}{1!} \left( \sum_{n=1}^m A_n(m) z^n \right) + \frac{1}{2!}
      \left( \sum_{n=1}^m A_n(m) z^n \right)^2 + \dots \label{eq:exp A}
  \end{align}
  The $k$-th term here can be expanded as
  \begin{align}
    \Biggl( \sum_{n=1}^m A_n & (m) z^n \Biggr)^k
      = \sum_{  i_1+\dots+i_m=k } \binom{k}{i_1,\dots,i_m} (A_1 z)^{i_1}(A_2 z^2)^{i_2} \dots (A_m z^m)^{i_m} \notag \\
    &= \sum_{ \substack{i_1+\dots+i_m=k\\ i_1+2i_2+\dots+mi_m \leq m} } \binom{k}{i_1,\dots,i_m}
      A_1^{i_1} \dots A_m^{i_m} z^{i_1+2i_2+\dots+mi_m} + \Oh(z^{m+1}) \notag \\
    &= \sum_{n=1}^m \left( \sum_{ \substack{\mathbf{i}\, \vdash n\\ i_1+\dots+i_n=k} }
       \binom{k}{i_1,\dots,i_n}  A_1^{i_1} \dots A_n^{i_n}\right) z^n + \Oh(z^{m+1}). \label{eq:binom}
  \end{align}
  Now~\eqref{eq:clf} follows from~\eqref{eq:exp A} and~\eqref{eq:binom}.
\end{proof}

\section{Asymptotics for Large Alphabet Size}\label{se:asympt m}

Next we derive the asymptotics of~$f_n(m)$ as~$m$, the cardinality of the alphabet,
tends to infinity. Define~$\kappa_n$ and $\mu_n = \mu_n(m)$ by
\[
  \sum_{n=0}^\infty \kappa_n z^n = \exp\left( \frac{z}{1-z} \right) \quad
  \text{and} \quad \sum_{n=0}^\infty \mu_n z^n = \phi(z).
\]
Note that $n!\kappa_n$ is Sloane's {\bf A000262} (several combinatorial interpretations are
given on that web page),
and that~$\kappa_n$ has the representation %TODO reference
\begin{equation}\label{eq:kappa}
  \kappa_n = \sum_{\mathbf{i}\, \vdash n} \frac{1}{i_1! \dots i_n!}, \qquad n\geq1.
\end{equation}
%
%The first values of~$n!\kappa_n$ are
%\[
%  1, 3, 13, 73, 501, 4051, 37633, 394353, 4596553, 58941091.
%\]
%
Then we can write
\begin{align*}
  f_n/m^n &= [z^n] \exp\left( \frac{z}{1-z} \right) \phi(z/m) \\
  &= \kappa_n + \kappa_{n-1} \mu_1/m + \dots + \kappa_0 \mu_n/m^n.
\end{align*}
If the dependence of~$\mu_n$ on~$m$ is not too strong, the first term on the right-hand
side should dominate when~$m\to\infty$. This is indeed the case:
\begin{theorem}
  If~$n\geq1$ is fixed and $m\to\infty$, we have
  \begin{equation}\label{eq:asympt m}
    f_n(m) \sim \kappa_n m^n.
  \end{equation}
\end{theorem}
\begin{proof}
  Since, as $m\to\infty$,
  \[
    A_j(m) = m^j + \Oh(m^{j/2}), \qquad j\geq1,
  \]
  we have
  \[
    A_j(m)^{i_j} = m^{j i_j}(1 + \Oh(m^{-j/2})),
  \]
  whence, for $\mathbf{i} \vdash n$,
  \[
    A_1(m)^{i_1} \dots A_n(m)^{i_n} = m^n (1 + \Oh(m^{-1/2})).
  \]
  The result thus follows from~\eqref{eq:clf} and~\eqref{eq:kappa}.
\end{proof}
Note that $\kappa_1=1$, $\kappa_2=\tfrac32$, and $\kappa_3=\tfrac{13}{6}$, in line with~\eqref{eq:first terms}.

\section{Asymptotics for Large Total Word Length}\label{se:asympt n}

\begin{theorem}\label{thm:asympt n}
For large total word length~$n$, the sequence $f_n=f_n(m)$ has the asymptotics
\begin{equation}\label{eq:asympt n}
  f_n \sim \frac{\phi(1/m)}{2\sqrt{\e \pi}} \times \frac{m^n \e^{2\sqrt{n}}}{n^{3/4}}, \qquad n\to\infty.
\end{equation}
More precisely, there is a full asymptotic expansion of the form
\begin{equation}\label{eq:full expans}
  f_n \sim \frac{\phi(1/m)}{2\sqrt{\e \pi}} \times \frac{m^n \e^{2\sqrt{n}}}{n^{3/4}}
    \left( 1 + \sum_{j\geq 1} c_j n^{-j/4} \right), \qquad n\to\infty.
\end{equation}
\end{theorem}
\begin{proof}
The proof of Theorem~\ref{thm:asympt n} is similar to the saddle point analysis~\cite[Example~VIII.7]{FlSe09}
of $\exp(z/(1-z))$, the ogf of~$\kappa_n$, slightly perturbed by the presence of the factor~$\phi(z)$.
The ogf~$F(z)$ is actually Hayman-admissible~\cite{FlSe09, Ha56}, but carrying out the saddle point method
explicitly gives access to a full asymptotic expansion, and will be useful for the refined
expansions required in Sections~\ref{se:joint}
and~\ref{se:law}.
Let us shift the dominating singularity from $z=1/m$ to $z=1$.
Then the integrand in Cauchy's formula
\begin{equation}\label{eq:cauchy}
  f_n = f_n(m) = \frac{m^n}{2\im \pi} \oint \frac{F(z/m)}{z^{n+1}} \dd z
\end{equation}
has an approximate saddle point at $z=\hat{z}:=1-1/\sqrt{n}$.
We write $z=\hat{z}\e^{\im\theta}$, where $\theta=\arg(z)$
is constrained by
\begin{equation}\label{eq:theta}
  |\theta| < n^{-\alpha}, \qquad \tfrac23 < \alpha < \tfrac34,
\end{equation}
so that~$z$ lies in a small arc around the saddle point. In this range we have the uniform expansion
\begin{align}
  z^{-n-1} &= \exp\left(-(n+1)\log\left(1-\frac{1}{\sqrt{n}}\right) - \im n \theta + \Oh(n^{-\alpha}) \right) \notag \\
  &= \exp\bigl( \sqrt{n} + \tfrac12 - \im n \theta + \Oh(n^{-1/2})\bigr), \qquad n\to\infty. \label{eq:z}
\end{align}
Furthermore
\[
  1-z = n^{-1/2} (1-\im\theta\sqrt{n} + \Oh(n^{-\alpha}) ),
\]
hence
\begin{equation}\label{eq:recipr}
  \frac{1}{1-z} = \sqrt{n} + \im \theta n - n^{3/2}\theta^2 + \Oh(n^{1/2-\alpha}).
\end{equation}
From~\eqref{eq:z} and~\eqref{eq:recipr} we get
\[
  z^{-n-1} \exp\left(\frac{z}{1-z}\right) = \exp\bigl( -\tfrac12 + 2\sqrt{n} - n^{3/2}\theta^2 \bigr)
  \times\bigl( 1 + \Oh(n^{1/2 - \alpha}) \bigr).
\]
Since $\phi(z/m)$ is analytic at $z=1$, the local expansion of the integrand in~\eqref{eq:cauchy}
at the saddle point~$\hat{z}$ is
\begin{equation}\label{eq:loc expans}
  \frac{F(z/m)}{z^{n+1}} = \phi(1/m) \exp\bigl(-\tfrac12 + 2\sqrt{n} - n^{3/2}\theta^2 \bigr)
  \times\bigl(1 + \Oh(n^{1/2 - \alpha})\bigr),
\end{equation}
valid as $n\to\infty$, uniformly w.r.t.~$\theta$ in the range~\eqref{eq:theta}.
Note that
\[
  \int_{-n^{-\alpha}}^{n^{-\alpha}} \e^{-n^{3/2}\theta^2} \dd \theta \sim \sqrt{\pi}n^{-3/4},
\]
so that integrating~\eqref{eq:loc expans} from~$-n^{-\alpha}$ to~$n^{-\alpha}$
yields the right-hand side of~\eqref{eq:asympt n}.
To prove~\eqref{eq:asympt n},
it remains to show that the integral from~$n^{-\alpha}$ to~$\pi$ grows slower
(the other half of the tail is handled by symmetry).
There is a~$C>0$ such that
\begin{equation}\label{eq:tail ineq}
  \left| \frac{F(z/m)}{z^{n+1}} \right| \leq C |z|^{-n} \exp \Re \left(\frac{1}{1-z}\right), \qquad |z|<1.
\end{equation}
If~$z=\hat{z}\e^{\im\theta}$ lies on the integration contour, then the factor~$|z|^{-n}$ in~\eqref{eq:tail ineq} is $\Oh(\e^{\sqrt{n}})$.
The remaining factor $\exp \Re (1/(1-z))$ decreases if~$|\theta|=|\arg(z)|$ increases, hence
\begin{align*}
  \int_{n^{-\alpha}}^\pi  \exp \Re \left(\frac{1}{1-\hat{z}\e^{\im \theta}}\right) \dd \theta 
  &\leq \pi \left. \exp \Re \left(\frac{1}{1-\hat{z}\e^{\im \theta}}\right)\right|_{\theta=n^{-\alpha}} \\
  &= \exp\bigl( \sqrt{n} - n^{3/2 - 2\alpha} + \Oh(1) \bigr).
\end{align*}
(The last line is obtained by recapitulating the derivation of~\eqref{eq:loc expans}, with $\theta=n^{-\alpha}$.)
Hence
\begin{equation}\label{eq:tail est}
  \left| \oint_{n^{-\alpha} < |\theta| < \pi} \frac{F(z/m)}{z^{n+1}} \dd z \right|
    \leq \exp\bigl( 2\sqrt{n} - n^{3/2 - 2\alpha} + \Oh(1) \bigr).
\end{equation}
This grows indeed slower than~$\e^{2\sqrt{n}}/n^{3/4}$, so that the proof of~\eqref{eq:asympt n}
is complete.

It remains to justify the full expansion~\eqref{eq:full expans}.
First note that it suffices to check that the central part $\int_{-n^{-\alpha}}^{n^{\alpha}}\dd\theta$
of the Cauchy integral has such an expansion,
as the tail estimate~\eqref{eq:tail est} lies asymptotically below~\eqref{eq:full expans}.
The full expansion of $z^{-n-1}$ proceeds by powers of~$n^{-1/2}$:
\[
  z^{-n-1} = \exp \left( \sqrt{n} + \tfrac12 - \im n \theta -\im \theta +
    \sum_{j\geq1} \frac{2(j+1)}{j(j+2)} n^{-j/2} \right).
\]
Taking more terms in~\eqref{eq:recipr} and in the expansion of the analytic function~$\phi$, we readily see that
the full local expansion of the integrand in~\eqref{eq:cauchy} is of the form
\[
  \frac{F(z/m)}{z^{n+1}} = \phi(1/m) \exp\bigl(-\tfrac12 + 2\sqrt{n} - n^{3/2}\theta^2\bigr)
  \times\left(1 + \sum_{k,j} c_{kj} n^{-k/2} \theta^j\right),
\]
where each term in the sum is~$\oh(1)$.
The resulting Gaussian integrals are of the kind
\begin{align*}
  \int_{-n^{-\alpha}}^{n^{-\alpha}} \theta^{M} \e^{-n^{3/2}\theta^2} \dd \theta &=
    n^{-3/4} \int_{-n^{3/4-\alpha}}^{n^{3/4-\alpha}} \left( \frac{y}{n^{3/4}} \right)^M \e^{-y^2} \dd y \\
  & \sim \mathrm{const} \times n^{-3(M+1)/4}, \qquad M\ \text{even}.
\end{align*}
(Those with odd~$M$ vanish.) This finishes the proof of~\eqref{eq:full expans}.
\end{proof}

\section{Joint Asymptotics}\label{se:joint}

Note that the limits $m\to\infty$ and $n\to\infty$ commute in the following sense:
Since we have $\kappa_n \sim 1/(2\sqrt{\e \pi}) \e^{2\sqrt{n}}/n^{3/4}$~\cite[Prop.~8.4]{FlSe09},
the right-hand side of~\eqref{eq:asympt m} has, as $n\to\infty$, the same asymptotics
as the right-hand side of~\eqref{eq:asympt n} for $m\to\infty$. We will now show that
letting~$m$ and~$n$ tend to infinity simultaneously yields the same result, regardless of
their respective speeds.

\begin{theorem}
  If both the word length and the alphabet size tend to infinity, we have
  \[
    f_n(m) \sim \frac{1}{2\sqrt{\e \pi}} \times \frac{m^n \e^{2\sqrt{n}}}{n^{3/4}}, \qquad m,n\to\infty.
  \]
\end{theorem}
\begin{proof}
  The result can be obtained by an adaption of the proof of Theorem~\ref{thm:asympt n}.
  Again we use Cauchy's formula, with the same saddle point contour as before:
  \begin{align}
    f_n(m) &= \frac{m^n}{2\pi}\hat{z}^{-n} \int_{-\pi}^{\pi}
      F(\hat{z}\e^{\im\theta}/m) \e^{-\im(n+1)\theta}\dd\theta \notag \\
    &= \frac{m^n}{2\pi}\hat{z}^{-n} \int_{-\pi}^{\pi}
      \exp\left(\frac{\hat{z}\e^{\im\theta}}{1-\hat{z}\e^{\im\theta}}\right)
      \phi\Bigl(\frac{\hat{z}\e^{\im\theta}}{m}; m\Bigr) \e^{-\im(n+1)\theta}\dd\theta. \label{eq:cauchy2}
  \end{align}
  We will show that
  \begin{equation}\label{eq:phi est}
    \phi\Bigl(\frac{\hat{z}\e^{\im\theta}}{m}; m\Bigr) \to 1, \qquad m,n\to\infty,\
    \text{uniformly w.r.t.}\ \theta\in{[-\pi,\pi]}.
  \end{equation}
  Assuming this we are done. Indeed, assertion~\eqref{eq:phi est} shows at the same time the validity
  of the local expansion~\eqref{eq:loc expans}, with $\phi(1/m)$ replaced by~$1$, and the persistence of
  the tail estimate~\eqref{eq:tail est}.
  
  To prove~\eqref{eq:phi est}, notice that
  \begin{equation}\label{eq:phi factor}
    \phi\Bigl(\frac{\hat{z}\e^{\im\theta}}{m}; m\Bigr)
      = \exp\left( \sum_{k=2}^\infty \frac{(-1)^{k-1}}{k}
      \frac{ m^{1-k} \hat{z}^k \e^{k\im\theta} }{ 1 - m^{1-k} \hat{z}^k \e^{k\im\theta} }\right).
  \end{equation}
  We have $|m^{1-k} \hat{z}^k \e^{k\im\theta}|<\tfrac12$ for $m\geq2$, hence
  \begin{align*}
    \left| \sum_{k=2}^\infty \frac{(-1)^{k-1}}{k} 
      \frac{ m^{1-k} \hat{z}^k \e^{k\im\theta} }{ 1 - m^{1-k} \hat{z}^k \e^{k\im\theta} }\right|
       &\leq \sum_{k=2}^\infty m^{1-k} \hat{z}^k \\
      &= \sum_{k=2}^\infty m^{1-k} \left(1-\frac{1}{\sqrt{n}}\right)^k \\      
      &= \frac{(1-1/\sqrt{n})^2}{m(1-1/m+1/(m\sqrt{n}))}.
  \end{align*}
  Thus the exponent in~\eqref{eq:phi factor} is uniformly~$\oh(1)$, which establishes~\eqref{eq:phi est}.
\end{proof}

\section{The Distribution of the Number of Words}\label{se:law}

A natural parameter to consider is the number~$W_n$ of words in a random finite language of total
word length~$n$.
(The alphabet size~$m\geq2$ is fixed throughout this section.)
The appropriate bivariate ogf, with~$z$ marking total word length and~$u$ marking number of words,
is given by
\begin{equation*}
  F(z,u) :=  \exp \left( \sum_{k=1}^\infty \frac{(-1)^{k-1}}{k} \frac{m z^k u^k}{1-mz^k} \right).
\end{equation*}
The expected number of words is then
\begin{equation}\label{eq:expec}
  \expec[W_n] = f_n^{-1} [z^n]\partial_u F(z,u)|_{u=1}.
\end{equation}
Notice that
\begin{equation}\label{eq:F F old}
  \partial_u F(z,u)|_{u=1} = F(z) \sum_{k=1}^\infty \frac{(-1)^{k-1}mz^k}{1-mz^k},
\end{equation}
so that the asymptotic analysis of~$[z^n]\partial_u F(z,u)|_{u=1}$ is an easy extension of
the one of~$f_n=[z^n]F(z)$ in Section~\ref{se:asympt n}: Close to the saddle point, the new factor
resulting from the right-hand side of~\eqref{eq:F F old} is
\[
  \frac{1}{1-z} = \sqrt{n}(1+\oh(1)).
\]
Hence $[z^n]\partial_u F(z,u)|_{u=1} \sim \sqrt{n}f_n$, so that, by~\eqref{eq:expec},
the expectation of~$W_n$ satisfies
\[
  \expec[W_n] \sim  \sqrt{n}, \qquad n\to\infty.
\]
Similarly, one can obtain the asymptotics $\sigma(W_n)\sim n^{1/4}/\sqrt{2}$ for
the standard deviation.

\begin{theorem}
  The number of words~$W_n$ in a random finite language admits a Gaussian limit law:
  \[
    \frac{W_n - a_n}{b_n} \to \mathcal{N}(0,1), \qquad n\to\infty,
  \]
  in distribution, where the scaling constants satisfy $a_n\sim\sqrt{n}$ and $b_n\sim n^{1/4}/\sqrt{2}$.
\end{theorem}
\begin{proof}
  As is well known, combinatorial limit laws can often
  be obtained by an asymptotic analysis of the probability generating function
  \begin{equation}\label{eq:pgf}
    \expec[u^{W_n}] = f_n^{-1} [z^n]F(z,u).
  \end{equation}
  Again, we adapt the proof of Theorem~\ref{thm:asympt n}.
  If~$u$ ranges in a fixed small neighbourhood of~$u=1$, the expansion~\eqref{eq:loc expans}
  generalizes to the uniform local expansion
  \[  
    \frac{F(z/m,u)}{z^{n+1}} = \phi(1/m,u;m) \exp\bigl(-\tfrac12 u + 2\sqrt{un} - u^{-1/2}n^{3/2}\theta^2 \bigr)
    \times\bigl(1 + \Oh(n^{1/2 - \alpha})\bigr),
  \]
  where
  \[
    \phi(z,u;m) := \exp \left( \sum_{k=2}^\infty \frac{(-1)^{k-1}}{k} \frac{mz^ku^k}{1-mz^k} \right).
  \]
  Integrating from $\theta=-n^{-\alpha}$ to~$n^{-\alpha}$, and taking into account~\eqref{eq:asympt n},
  we infer that~\eqref{eq:pgf} has the uniform asymptotics
  \[
    \expec[u^{W_n}] \sim \exp(h_n(u)), \qquad n\to\infty,
  \]
  with
  \[
    h_n(u) := 2(\sqrt{u}-1)\sqrt{n} + \tfrac14 \log u + \log \frac{\phi(1/m,u;m)}{\phi(1/m;m)}.
  \]
  Note that, for $n\to\infty$,
  \begin{align*}
    h_n'(1) &= \sqrt{n} + \Oh(1), \\
    h_n''(1) &= -\tfrac12 \sqrt{n} + \Oh(1), \\
    h_n'''(1) &= \tfrac34 \sqrt{n} + \Oh(1), \\
  \end{align*}
  so that the function~$h_n(u)$ satisfies the conditions of~\cite[Theorem~9.13]{FlSe09}, itself
  taken from~\cite{Sa78}. We conclude that
  \[
    \frac{W_n - h_n'(1)}{(h_n'(1) + h_n''(1))^{1/2}}
  \]
  converges in distribution to a standard normal random variable.
\end{proof}

\bibliographystyle{siam}
\bibliography{../gerhold}

\end{document}